\documentclass[final,1p,times]{elsarticle}
\usepackage{lineno,hyperref}
\modulolinenumbers[5]

\usepackage{amsfonts}
\usepackage{amsmath}
\usepackage{amsthm}
\usepackage{amssymb}
\usepackage{graphicx}
\usepackage{color}
\usepackage{rotating}
\usepackage{url}
\usepackage{epstopdf}

\usepackage{color}

\DeclareMathOperator*{\interior}{interior}

\newcommand{\rec}[1]{\textup{rec.cone}(#1)}
\newcommand{\recdual}[1]{\textup{rec.cone}^*(#1)}
\newcommand{\subfcts}[1]{\mathcal{F}_{#1}}

\newtheorem{proposition}{Proposition}
\newtheorem{theorem}{Theorem}

\newtheorem{lemma}{Lemma}
\newtheorem{remark}{Remark}

\journal{Discrete Optimization}

\begin{document}
\begin{frontmatter}

%Guidelines: https://www.elsevier.com/journals/discrete-optimization/1572-5286/guide-for-authors#1001

\title{Some cut-generating functions for second-order conic sets}


\author[mainaddress]{Asteroide Santana\corref{mycorrespondingauthor}
%\fnref{Aster_footnote}
}
%\fntext[Aster_footnote]{asteroide.santana@gatech.edu}
\ead{asteroide.santana@gatech.edu}

\author[mainaddress]{Santanu S. Dey \corref{mycorrespondingauthor}
%\fnref{Santanu_footnote}
}
%\fntext[Santanu_footnote]{santanu.dey@isye.gatech.edu}
\ead{santanu.dey@isye.gatech.edu}

\cortext[mycorrespondingauthor]{Corresponding authors}



\address[mainaddress]{ISyE, Georgia Institute of Technology, 765 Ferst Drive NW, Atlanta, GA 30332-0205, USA}

\begin{abstract}
In this paper, we study cut generating functions for conic sets. Our first main result shows that if the conic set is bounded, then cut generating functions for integer linear programs can easily be adapted to give the integer hull of the conic integer program. Then we introduce a new class of cut generating functions which are non-decreasing with respect to second-order cone. We show that, under some minor technical conditions, these functions together with integer linear programming-based functions are sufficient to yield the integer hull of intersections of conic sections in $\mathbb{R}^2$.
\end{abstract}


\end{frontmatter}

%\linenumbers

\section{Introduction: Subadditive dual of conic integer programs}

A natural generalization of linear integer programming is \emph{conic integer programming}.  Given a  \emph{regular} cone $K \subseteq \mathbb{R}^n$, that is a cone that  is pointed, closed, convex, and full dimensional, we can define a conic integer program as:
\begin{eqnarray}\label {eq:conicIP}
\begin{array}{rl}
\textup{inf} &c^{\top} x\\
\textup{s.t.} & Ax - b \in K\\
& x \in \mathbb{Z}^n_+,
\end{array}
\end{eqnarray}
where $A \in \mathbb{R}^{m \times n}, c\in \mathbb{R}^n$ and $b \in \mathbb{R}^m$. As is standard, we will henceforth write the constraint $Ax - b \in K$ as $Ax  \succeq_K b$, where we use the notation that $u\succeq_K v$ if and only if $u - v\in K$. In the case where $K$ is the non-negative orthant, that is $K = \mathbb{R}^m_{+}$, the conic integer program is a standard linear integer program.

A natural way to generate cuts for conic integer programs is via the notion of cut-generating functions~\cite{ConfortiCDLM15}. Consider a function $f: \mathbb{R}^m \rightarrow \mathbb{R}$ that satisfies the following:
\begin{enumerate}
\item $f$ is \emph{subadditive}, that is $f (u) + f(v) \geq f(u + v)$ for all $u, v \in \mathbb{R}^m$,
\item $f$ is \emph{non-decreasing with respect to $K$}, that is $f(u) \geq f(v)$ whenever $u \succeq_K v$,
\item $f(0) = 0$.
\end{enumerate}
Then it is straightforward to see that the inequality 
\begin{eqnarray}\label{eq:cut}
\sum_{j = 1}^n f(A^j) x_j \geq f(b),
\end{eqnarray}
is valid for the conic integer program (\ref{eq:conicIP}), where $A^j$ is the $j$-th column of $A$. We denote the set of functions satisfying (1.), (2.) and (3.) above as $\subfcts{K}$.

%Without assuming nonnegativity of the variables the validity of inequality (\ref{eq:cut}) still holds provided that, in addition to (1.), (2.) and (3.),  $f$ satisfies 
%\begin{enumerate}
%\item[4.] $f(v)=-f(-v)$ for all $v$ in $\mathbb{R}^m$.
%\end{enumerate}

In the paper~\cite{DiegoSantanu2012}, it was shown that, assuming a technical `\emph{discrete Slater}' condition 
%($\textup{interior}(K) \cap \{ y\,|\, y = Ax - b, x \in \mathbb{Z}_+^n\} \neq \emptyset$)
holds, the closure of the convex hull of the set of integer feasible solutions to (\ref{eq:conicIP}) is described by inequalities of the form (\ref{eq:cut}) obtained from $\subfcts{K}$. This result from~\cite{DiegoSantanu2012} generalizes result on subadditive duality of linear integer programs~\cite{Jeroslow79,JOHNSON1973,Johnson1979,WOLSEY1981}, that is inequalities (\ref{eq:cut}) give the convex hull of (\ref{eq:conicIP}) when $K =\mathbb{R}^m_{+}$ and the constraint matrix $A$ is rational. Also see~\cite{Karzan2015,Kılınç-Karzan2016} for related models and results.

In the case where $K = \mathbb{R}^m_{+}$ and assuming $A$ is rational, a lot more is known about the subset of functions from  $\subfcts{\mathbb{R}^m_{+}}$ that are sufficient to describe the convex hull of integer solutions (also called as the integer hull). For example, these functions have a constructive characterization using the Chv\'atal-Gomory procedure~\cite{BlairJ82}, it is sufficient to consider functions that are applied to every $2^n$ subset of constraints at a time (see \cite{Schrijverbook1}, Theorem 16.5), or for a fixed $A$ there is a finite list of functions independent of $b$ that describes the integer hull~\cite{WOLSEY1981}.

The main goal of this paper is to similarly better understand structural properties of subsets of functions from $\subfcts{K}$ that are sufficient to produce the integer hull of the underlying conic representable set $\{x\in\mathbb{R}^n\,|\, Ax   \succeq_K b \}$.

\section{Main results} 
We will refer to the \emph{dual cone} of a cone $K$ as $K^*$ which we remind the reader is the set $K^*:=\{y\in\mathbb{R}^m\,|\, y^{\top}x \geq 0 \ \forall x \in K\}.$ Given a positive integer $m$, we denote the set $\{1, \dots, m\}$ by $[m]$. And given a subset $X$ of $\mathbb{R}^n$ we denote its integer hull by $X^I$. 

\subsection{Bounded sets}

Given a regular cone $K$ we call as \textit{linear composition} the set of functions $f$ obtained as follows: Let the vectors $w^1, w^2, \dots,w^p\in K^*$ and the function $f:\mathbb{R}^m\to\mathbb{R}$ be given by
\begin{align}
f(v)=g((w^1)^{\top}v, (w^2)^{\top}v, \cdots, (w^p)^{\top}v),\label{Linearcomposite}
\end{align}
where $g\in\subfcts{\mathbb{R}^p_+}$ satisfies $g(u) = -g(-u)$ for all $u\in\mathbb{R}^p$.
It is straightforward to see that linear composition functions belong to $\subfcts{K}$ and also satisfy  $f(v) = -f(-v)$ for all $v\in\mathbb{R}^m$, which implies that $f$ generates valid inequalities of the form (\ref{eq:cut}) even when the variables are not required to be non-negative. Our first result describes a class of conic sets for which linear composition functions are sufficient to produce the convex hull.
\begin{theorem}\label{thm:PolyhedralOuterApproxOfBoundedConicCuts}
Let $K\subseteq \mathbb{R}^m$ be a regular cone. Consider the conic set
$
T=\{x\in\mathbb{R}^n\, | \ Ax\succeq_K b\},
$
where $A\in\mathbb{R}^{m\times n}$ and $b\in\mathbb{R}^m$.
Assume $T$ has nonempty interior. Let $\pi^{\top}x \geq \pi_{0}$ be a valid inequality for $T^I$ where $\pi\in\mathbb{Z}^n$ is non-zero. Assume $B := \{x\in T\, | \ \pi^{\top}x \leq \pi_0\}$ is nonempty and bounded. Then, for some natural number $p \leq 2^n$, there exist vectors $y^1,y^2,\dots,y^{p}\in K^*$ such that $\pi^{\top}x \geq \pi_0$ is a valid inequality for the integer hull of the polyhedron
$
Q=\{x\in\mathbb{R}^n\,|\, (y^i)^{\top}A x \geq (y^i)^{\top}b, \ i\in [p]\},
$
where $(y^i)^{\top}A$ is rational for all $i\in[p]$.
\end{theorem}
We highlight here that particular care was taken in Theorem~\ref{thm:PolyhedralOuterApproxOfBoundedConicCuts}
to ensure that the outer approximating polyhedron has rational constraints.

Since a valid inequality for $Q^I$ can be obtained using a subadditive function $g\in\subfcts{{\mathbb{R}}^p_{+}}$ that satisfies $g(u) = -g(-u)$ for all $u\in\mathbb{R}^p$ \cite{WolseyNemhouser1990} (note that the constraints matrix defining $Q$ is rational), Theorem \ref{thm:PolyhedralOuterApproxOfBoundedConicCuts} implies that if a cut separates a bounded set from $T$, then it can be obtained using exactly one function (\ref{Linearcomposite})
with $p \leq 2^n$. Geometrically, Theorem~\ref{thm:PolyhedralOuterApproxOfBoundedConicCuts}  can be interpreted as the fact that if the set of points separated is bounded, then 	the cut can be obtained using a rational polyhedral outer approximation. 

We obtain the following corollary immediately: 
If the set $\{x\in\mathbb{R}^n \,|\, Ax \succeq_K b\}$ is compact and has non-empty interior, then it is sufficient to restrict attention to linear composition functions  to obtain the convex hull. %Geometrically, this result can be interpreted as the fact that it is possible to use polyhedral outer-approximation and then use linear integer programming cuts on the resulting outer approximation to obtain the integer hull for compact sets. 
A proof of Theorem~ \ref{thm:PolyhedralOuterApproxOfBoundedConicCuts} is presented in Section \ref{sec:result1}.
 
\subsection{New family of cut-generating functions} 

In the previous section we stated that any valid inequality for the integer hull of a bounded conic set can be obtained using linear composition functions.
So what happens when the underlying set is not bounded? Consider the simple unbounded set
$T'=\{(x_1,x_2)\in\mathbb{R}^2_+\,| \ x_1x_2\geq 1\},$
which is one branch of a hyperbola\footnote{In this paper, we refer to the curve, as well as the convex region delimited by this curve, as the branch of a hyperbola. Same for parabolas and ellipses.}. This set is conic representable, that is $T'=\{x\in\mathbb{R}^2_+\,| \  Ax\succeq_K b\}$, where $K$ is the second-order cone $\mathcal{L}^3$ and 
\begin{eqnarray}\label{eq:AT}
A=\left[\begin{array}{rl}
0 & 0 \\ 
1 & -1 \\ 
1 & 1
\end{array}\right], \ \ \ 
b=\left[\begin{array}{c} -2\\ 0\\ 0 \end{array}\right].
\end{eqnarray}
(We use the notation $\mathcal{L}^m :=\left\{x\in\mathbb{R}^m\,| \ \sqrt{x_1^2+x_2^2+\dots+x_{m-1}^2}\leq x_m \right\}$ to represent the second-order cone in $\mathbb{R}^m$.)
The integer hull of $T'$ is given by the following two inequalities:
\begin{align}
x_1\geq 1, \ x_2\geq 1. \label{ineq:FacetsHyperbola}
\end{align}
It is straightforward to verify that the inequalities (\ref{ineq:FacetsHyperbola}) are \emph{not valid for any polyhedral outer approximation of $T'$}. Indeed any polyhedral outer approximation of $T'$ contains integer points not belonging to $T'$ (see Proposition \ref{prop:NoGoodOuterApproximationForHyperbola}). 
Therefore, applying the cut-generating recipe (\ref{Linearcomposite}) a finite number of times (that is considering integer hulls of a finite number of polyhedral outer approximations of $T'$) does not yield $x_1 \geq 1$. However, we note here that we can use linear composition (\ref{Linearcomposite}) to obtain a cut of the form $x_1 + x_2/k \geq 1$ where $k \in \mathbb{Z}_{+}$ and $k \geq 1$. Clearly 
$$
\bigcap_{k \in \mathbb{Z}_{+}, k \geq 1}\left\{x \in \mathbb{R}^2 \,|\, x_1 + x_2/k \geq 1\right\} = \{x\in \mathbb{R}^2 \,|\, x_1 \geq 1\}.
$$
However, it would be much nicer if we could \emph{directly obtain} $x_1 \geq 1$ without resorting to obtaining it as an implication of an infinite sequence of cuts. 

Many papers \cite{gomory:jo:1972a,gomory:jo:2003,dash:gu:2006,dey:ri:2008,dey:ri:2007b,RichardLM09,RichardD10,Basu2016,KoeppeZ2015a} have explored various families of subadditive functions for linear integer programs. Our second result, in the same spirit, is a parametrized family of functions that belongs to $\subfcts{K}$, where $K$ is the second-order cone $\mathcal{L}^m$. The formal result is as follows:

\begin{theorem}\label{thm:newfunction}
Let $j \in [m-1]$. Define $\Gamma_j:=\{\gamma\in\mathbb{R}^m\, | \ \gamma_m\geq \sum_{i=1}^{m-1}|\gamma_i|, \ \gamma_m>|\gamma_j|\}$. Suppose $\gamma\in\Gamma_j\cup \interior{(\mathcal{L}^m)}$. Consider the real-valued function $f_{\gamma} : \mathbb{R}^m \rightarrow \mathbb{R}$ defined as:
\begin{align}
f_{\gamma}(v)= \begin{cases}
\gamma^{\top}v +1 \ & \text{if} \ v_j\neq 0 \ \text{and} \ \gamma^{\top}v\in\mathbb{Z}, \vspace{0.1cm}\\
\left\lceil\gamma^{\top}v\right\rceil    &\text{otherwise}.  
\end{cases}\label{OurFormula}
\end{align}
Then, $f_{\gamma}\in \subfcts{\mathcal{L}^m}$.
\end{theorem}

\begin{figure}[h]
\begin{center}
\includegraphics[scale=0.5]{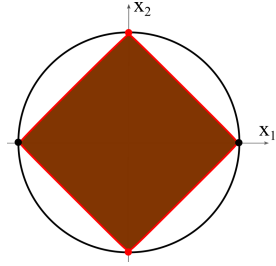} 
\end{center}
%\caption{Horizontal slice of the second-order cone $\mathcal{L}^3=\{x\in\mathbb{R}^3: \ \sqrt{x_1^2+x_2^2}\leq x_3\}$ and $\Gamma_1=\{x\in\mathbb{R}^3\, | \ x_3\geq |x_1|+|x_2|, \ x_3>|x_1|\}$.}
\caption{Slice at $x_3=1$ of the second-order cone $\mathcal{L}^3$ and $\Gamma_1$.}\label{fig:Gammaj}
\end{figure}

To see an example of use of $f_{\gamma}$, consider $j=1$ and $\gamma = ( 0, 0.5, 0.5)$. Then applying the resulting function $f_{\gamma}$ to the columns of (\ref{eq:AT}) we obtain the inequality $x_1 \geq 1$. 

Note that the validity of  the first inequality in (\ref{ineq:FacetsHyperbola}) can be explained via the disjunction $x_1 \leq 0 \vee x_1 \geq 1$. Therefore, some of the cuts generated using (\ref{OurFormula}) can be viewed as split disjunctive cuts. Significant research has gone into describing split disjunctive cuts (newer implied conic constraints) for conic sections~\cite{DadushDV11,BelottiGPRT13, ModaresiKV15,ModaresiV2015,KarzanY15,YildizC15,Burer2016}. However, to the best of our knowledge, there is no family of subadditive functions in $\subfcts{\mathcal{L}^m}$ which have been described in closed form previously. 

It is instructive to compare cuts obtained using (\ref{OurFormula}) with two well-known approaches for generating cuts for the integer hull of second-order conic sets~\cite{CezikIyengar2005,AtamturkNarayanan2008}. Note that the CG cuts described in~\cite{CezikIyengar2005} are a special case\footnote{More precisely, in~\cite{CezikIyengar2005} the variables are assumed to be non-negative, in which case we can drop the requirement of $g$ satisfying $g(u)=-g(-u)$ in the definition of linear composition.} of cuts generated via linear composition (\ref{Linearcomposite}). Therefore as discussed above, the CG cuts described in~\cite{CezikIyengar2005} cannot generate (\ref{ineq:FacetsHyperbola}) directly. The conic MIR procedure described in~\cite{AtamturkNarayanan2008} begins with first generating an extended formulation which applied to $T'$ would be of the form:
\begin{eqnarray*}
t_0 &\leq & x_1 + x_2 \\
t_1 &\geq & 2\\
t_2 &\geq &|x_1 - x_2|\\
t_0 &\geq &|\!|t|\!|_2\\
&&x_1, x_2 \in \mathbb{Z}_{+}, t \in \mathbb{R}^3_{+}.
\end{eqnarray*}
Then, cuts for the set $\{(x, t_2)\in \mathbb{Z}^2_{+} \times \mathbb{R}\,|\,  t_2 \geq |x_1 - x_2|\}$ are considered. However, this set is integral in this case and therefore no cuts are obtained. Thus, the conic MIR procedure does not generate the inequalities (\ref{ineq:FacetsHyperbola}).

\begin{remark}\label{remark:OurFunctionIsNotNondec.w.r.t.R3}
The function $f_{\gamma}$ defined in (\ref{OurFormula}) is piecewise linear, and it is therefore tempting to think it may also belong to $\subfcts{\mathcal{R}^m_{+}}$. However it is straightforward to check that $f_{\gamma}$ is \text{not} necessarily non-decreasing with respect to $\mathbb{R}_+^3$. Let $j=1$ and $\gamma=(0,\rho,\rho)$ where $\rho$ is a positive scalar. Then
\begin{align*}
f_{\gamma}(v_1,v_2,v_3)= \begin{cases}
\rho(v_2+v_3) +1 \ & \text{if} \ v_1\neq 0 \ \text{and} \ \rho(v_2+v_3)\in\mathbb{Z}, \vspace{0.1cm}\\
\left\lceil\rho(v_2+v_3)\right\rceil    &\text{otherwise}.  
\end{cases} %\label{OurParticularFormula}
\end{align*}
Consider the vectors  $u=(0,0,1/\rho)$ and $v=(-1,0,1/\rho)$. Then $u\geq_{\mathbb{R}^3_+}v$, whereas
$f_{\gamma}(u)=1<2=f_{\gamma}(v).$ \end{remark}

A proof of Theorem~\ref{thm:newfunction} is presented in Section~\ref{sec:result2}.

\subsection{Cuts for integer conic sets in $\mathbb{R}^2$}\label{sec:cuts}

As mentioned earlier, the family of functions (\ref{OurFormula}) yields the inequalities (\ref{ineq:FacetsHyperbola}). Indeed, we are able to verify a more general result in $\mathbb{R}^2$. 
To explain this result, we will need the following results:
\begin{lemma}\label{lemma:ConicRepresentationHyperbola}
 Let $G$ be one branch of a hyperbola in $\mathbb{R}^2$. Then $G$ can be represented as
$
G=\{x\in\mathbb{R}^2\, | \ Ax\succeq_{\mathcal{L}^{3}}b\},
$
where $A\in\mathbb{R}^{3\times 2}$ is such that $A_{11},A_{12}=0$. Moreover, the asymptotes of $G$ have equations
\begin{align}
(A_{21}+A_{31})x_1+(A_{22}+A_{32})x_2&=b_3+b_2\label{eq:aa61}\\
(-A_{21}+A_{31})x_1+(-A_{22}+A_{32})x_2&=b_3-b_2.\label{eq:aa6}
\end{align}
\end{lemma}
In order to generate cuts for $G$ in Lemma~\ref{lemma:ConicRepresentationHyperbola} using functions (\ref{OurFormula}) we first require the variables to be non-negative. Therefore, let us write $G$ as 
\begin{align}
A^1x^+_1-A^1x^-_1+A^2x^+_1-A^2x^-_1&\succeq_{\mathcal{L}^{3}} b\label{eq:aa1}\\
x_1^+,x_1^-,x_2^+,x_2^-&\geq 0 \label{eq:aa2}\\
x_j=x_j^+-x_j^- \ j&\in\{1,2\}.\label{eq:aa3}
\end{align}
Assuming that the asymptotes of $G$ are rational, we may assume that the coefficients in (\ref{eq:aa61}) and (\ref{eq:aa6}) are integers and then let  
$\tau=\gcd(A_{21}+A_{31}, A_{22}+A_{32}).$
Let $j=1$ and $\gamma=\left(0,1/\tau,1/\tau\right)$. Then we apply the function $f_{\gamma}$ to obtain the following cut for (\ref{eq:aa1}), (\ref{eq:aa2}):
\begin{align}
\frac{(A_{21}+A_{31})}{\tau}x^+_1-\frac{(A_{21}+A_{31})}{\tau}x^-_1+\frac{(A_{22}+A_{32})}{\tau}x^+_2-\frac{(A_{22}+A_{32})}{\tau}x^-_2\geq f_{\gamma}(b).\label{eq:aa4}
\end{align}
Now, using (\ref{eq:aa3}) and observing that the coefficient of $x_j^+$ is the negative of the coefficient of $x^-_j$ in (\ref{eq:aa4}), $j=1,2$, we can project the inequality (\ref{eq:aa4}) to the space of the original $x$ variables. The resulting cut is parallel to the asymptote (\ref{eq:aa61}). We can do a similar calculation to obtain a cut parallel to the other asymptote (\ref{eq:aa6}). We state all this concisely in the next proposition.
\begin{proposition}\label{prop:1}
Let
$G=\{x\in\mathbb{R}^2\, | \ Ax\succeq_{\mathcal{L}_+^{3}}b\}$
be one branch of a hyperbola with rational asymptotes, where $A\in\mathbb{R}^{3\times 2}$ and $A_{11},A_{12}=0$. Then the following inequalities are valid for $G^I$:
\begin{align}
(u^j)^{\top}A^1x_1+(u^j)^{\top}A^2x_2\geq \tau^jf_{\gamma^j}(b), \label{CutsForHyperbola}
\end{align}
where $u^1=(0,1,1), \ u^2=(0,-1,1)$, $\tau^j=\gcd((u^j)^{\top}A^1,(u^j)^{\top}A^2)$ and $\gamma^j:=u^j/\tau^j$, $j=1,2$.
\end{proposition}

We are now ready to state the main result of this section.

\begin{theorem}\label{thm:2d}
Let 
$
W=\displaystyle\bigcap_{i\in[m]}W^i,
$
where 
$
W^i=\{x\in\mathbb{R}^2\, | \ A^ix\succeq_{\mathcal{L}^{m_i}}b^i\},
$
$A^i\in\mathbb{R}^{m_i\times 2}$, $b^i\in\mathbb{R}^{m_i}$ and $\mathcal{L}^{m_i}$ is the second-order cone in $\mathbb{R}^{m_i}$.
Assume $W$ has nonempty interior and each constraint $A^ix\succeq_{\mathcal{L}^{m_i}}b^i$ in the description of $W$ is either a half-space or a single conic section, such as a parabola, an ellipse, or one branch of a hyperbola. Also assume that if $W^i$ is a hyperbola, then it is non-degenerate and it is written as in Lemma~\ref{lemma:ConicRepresentationHyperbola}, that is $A^i\in\mathbb{R}^{3\times 2}$ and $A^i_{11},A^i_{12}=0$. Finally, we assume that each $W^i$ is non-redundant, that is, for all $j\in[m]$, $W$ is strictly contained in $\displaystyle\bigcap_{i\in[m],i\neq j}W^i$. Then the following statements hold:
\begin{enumerate}
\item If $W\cap\mathbb{Z}^2=\emptyset$, then this fact can be certified with the application of at most two inequalities generated from (\ref{Linearcomposite}) or (\ref{CutsForHyperbola});
\item Assume $\textup{interior}(W)\cap\mathbb{Z}^2\neq \emptyset$. If $\pi^{\top}x\geq \pi_0$ defines a face of $W^I$ where $\pi\in\mathbb{Z}^2$ is non-zero, then this inequality can be obtained with application of exactly one function (\ref{Linearcomposite}) or it is one of the inequalities (\ref{CutsForHyperbola}).
\end{enumerate}
\end{theorem}

Proof of Lemmma~\ref{lemma:ConicRepresentationHyperbola} and Theorem~\ref{thm:2d} are presented in Section~\ref{sec:result3}.

\section{Cutting-planes separating bounded set of points}\label{sec:result1}
In this section, we prove Theorem~\ref{thm:PolyhedralOuterApproxOfBoundedConicCuts}. We begin by stating three well-known lemmas.
\begin{lemma}\label{lemma:InteriorOfDualCone}
Let $K\in\mathbb{R}^n$ be a closed cone and let $K^*$ denote its dual. Then
$\interior{(K^*)}= \{y\in\mathbb{R}^n\, | \ y^{\top}x>0 \ \forall x\in K\setminus\{0\}\}.$
\end{lemma}
Hereafter, we will denote the recession cone of a set $C$ by $\rec{C}$ and the dual of $\rec{C}$ by $\recdual{C}$.
\begin{lemma}\label{lemma:ConditionBoundedConicProgram}
Let $C\subseteq \mathbb{R}^n$ be a nonempty closed convex set. 
Then the following statements hold:
\begin{itemize}
\item[(i)] for every $c\in \interior{(\recdual{C})}$ the problem $\inf\{c^{\top}x \,|\, x\in C\}$ is bounded. 
\item[(ii)] for every $c
\notin \recdual{C}$ the problem $\inf\{c^{\top}x \,|\, x\in C\}$ is unbounded.
\end{itemize}
\end{lemma}
%%%%%%%%%%
%%%%%%%%%%
\begin{lemma}[Conic strong duality \cite{NemirovskiNotes}]\label{lem:strdual}
Let $K\subseteq \mathbb{R}^m$ be a regular cone. Consider the conic set
$
T=\{x\in\mathbb{R}^n\, | \ Ax\succeq_K b\},
$
where $A\in\mathbb{R}^{m\times n}$ and $b\in\mathbb{R}^m$.
Assume $\interior{T}\neq \emptyset$. If $c\in \mathbb{R}^n$ is such that $\inf\{c^{\top}x \,|\, x\in T\}$ is bounded, then there exists $y \in K^{*}$ such that $y^\top A = c^\top$ and $y^\top b = \inf\{c^{\top}x \,|\, x\in T\}$.
\end{lemma}
%%%%%%%%%%
%%%%%%%%%%
The next lemma states that under some conditions it is possible to separate a point from a set using a rational separating hyperplane. 
\begin{lemma}\label{lemma:RationalSeparHyperplane}
Let $C\subseteq \mathbb{R}^n$ be a closed convex set. Assume $\interior{(\recdual{C})}\neq \emptyset$. Let $z\notin C$. Then there exist $\pi\in\mathbb{Q}^n$, $\pi\neq 0$, and $\pi_0\in\mathbb{R}$ such that $\pi^{\top}z <\pi_0\leq \pi^{\top}x$ for all $x\in C$.
\end{lemma}
\begin{proof}
The standard separation theorem ensures that there exist $w\in\mathbb{R}^n$, $w\neq 0$, and $w_0\in\mathbb{R}$ such that  $w^{\top}z <w_0\leq w^{\top}x$ for all $x\in C$. As $\interior{(\recdual{C})}\neq \emptyset$ there exist $w^1,w^2,\ldots, w^{n+1}\in \interior{(\recdual{C})}$ affinely independent. For every $i\in[n+1]$ let $w^i_0=\inf\{(w^i)^{\top}x\, | \ x\in C\}$. In view of Lemma~\ref{lemma:ConditionBoundedConicProgram} we have that  $w^i_0$ is finite for all $i\in[n+1]$. Since $w_0-w^{\top}z>0$ and $z$ is fixed, we can chose $\varepsilon_i>0$, $i\in[n+1]$, such that 
\begin{align}
| \sum_{i=1}^{n+1} \varepsilon_i(w^i)^{\top}z - \sum_{i=1}^{n+1} \varepsilon_iw_0^i | < w_0-w^{\top}z.\label{ineq:wz}
\end{align}
Moreover, since $w^1,w^2,\dots, w^{n+1}$ are affinity independent, the cone generated by these vectors is full dimensional. Thus, the scalars $\varepsilon_i>0$, $i\in[n+1]$, can be chosen such that
$\pi:=w+\sum_{i=1}^{n+1} \varepsilon_iw^i\in\mathbb{Q}^n$. Now observe that
\begin{align*}
\pi^{\top}z < w_0 +\sum_{i=1}^{n+1}\varepsilon_i w^i_0 &\leq \inf\{w^{\top}x \, | \ x\in C\} + \sum_{i=1}^{n+1} \inf\{(\varepsilon_i w^i)^{\top}x \, | \ x\in C\}\\
&\leq \inf\{(w^{\top} + \sum_{i=1}^{n+1} \varepsilon_i w^i)^{\top}x \, | \ x\in C\}\leq \pi^{\top}x \ \ \forall x\in C,
\end{align*}
where the first strict inequality follows from (\ref{ineq:wz}). Therefore, $\pi^{\top}z < \pi_0 \leq \pi^{\top}x$ for all $x\in C$, where $\pi_0:=w_0 +\sum_{i=1}^{n+1}\varepsilon_i w^i_0$.
\end{proof}
The next result will imply Theorem \ref{thm:PolyhedralOuterApproxOfBoundedConicCuts}.
\begin{proposition}\label{prop:PolyhedralOuterApproxOfBoundedConicCuts}
Let $T$ be the set as in the statement of Lemma~\ref{lem:strdual}. Consider the set
$
B:=\{x\in T\, | \ \pi^{\top}x\leq \pi_0\},
$
where $\pi\in\mathbb{Z}^n$ is non-zero. Then $B$ is bounded if and only if $\pi\in\textup{interior}(\recdual{T})$, in which case for some natural number $p'$, there exist vectors $y^1,y^2\dots,y^{p'}\in K^*$ such that the polyhedron
\begin{align*}
P=\{x\in\mathbb{R}^n \,| \ \pi^{\top}x\leq \pi_0, \ (y^i)^{\top}A x \geq (y^i)^{\top}b, \ i\in [p']\}
\end{align*}
contains $B$ and $P^I=B^I$, where $(y^i)^{\top}A$ is rational for all $i\in[p']$.
\end{proposition}
\begin{proof} 
Assume $B$ is bounded.
We claim that
$d^{\top}\pi>0, \ \text{for all} \ d\in \rec{T}\setminus\{0\}$.
Indeed, if $d\in\rec{T}$ is such that $d^{\top}\pi\leq 0$, then $d\in\rec{B}$, which implies that $d=0$ since $B$ is bounded.
Now, in view of Lemma~\ref{lemma:InteriorOfDualCone}, the claim implies that $\pi\in \interior{(\recdual{T})}$.

Assume $\pi\in\interior{(\recdual{T})}$.
As $\pi \in \mathbb{Z}^n$, let $\{v^1,v^2,\dots,v^{n-1}\} \subseteq \mathbb{Q}^n$ be an orthogonal basis of the linear subspace orthogonal to $\pi$. Since $\pi\in\interior{(\recdual{T})}$, there exists a positive constant $\varepsilon$ such that $w^i:=\pi+\varepsilon v^i$ and $w^{i+n-1}:=\pi-\varepsilon v^i$ belong to $\interior{(\recdual{T})}$ for all $i\in[n-1]$. As we may assume that $\varepsilon$ is rational, we obtain that $w^i$ is rational for all $i\in[2n-2]$. It follows from Lemma~\ref{lemma:ConditionBoundedConicProgram} and Lemma~\ref{lem:strdual} that for all $i\in[2n-2]$ there exists $y^i\in K^*$ such that
$(y^i)^{\top}Ax\geq (y^i)^{\top}b$ is a valid inequality for $T$, where $(y^i)^{\top}A = w^i \in \mathbb{Q}^n$. 
Since $\pi\in \interior{(\recdual{T})}$, Lemma~\ref{lemma:ConditionBoundedConicProgram} and Lemma~ \ref{lem:strdual} also imply that there exists $y^{2n-1} \in K^{*}$ such that $(y^{2n-1})^{\top}Ax \geq (y^{2n-1})^{\top}b$ is a valid inequality for $T$, where $(y^{2n-1})^{\top}A = \pi^{\top} \in \mathbb{Q}^n$. 
Now, let $P^1=\{x\in\mathbb{R}^n\, | \ \pi^{\top}x\leq \pi_0, \ (y^i)^{\top}A x\geq (y^i)^{\top}b, \ i\in[2n-1]\}$.
By our choice of $w^i$ and using the fact that $(y^{2n-1})^{\top}b\leq \pi^{\top}x\leq \pi_0$ for all $x\in P^1$ (if  $\pi_0\leq (y^{2n-1})^{\top}b$, then $P^1=\emptyset$), it is easy to verify that $P^1$ is bounded. Since $P^1$ contains $B$, we obtain that $B$ is also bounded.

If $(P^1)^I=B^I$, then we are done by setting $P$ to $P^1$, in which case $p'=2n-1$. Otherwise, as $P^1$ is bounded, there is only a finite number of integer points $z\in P^1\setminus B$. 
For each one of these points $z$, we construct a rational valid inequality $w_0\leq w^{\top}x$ for $T$ that is guaranteed by Lemma~\ref{lemma:RationalSeparHyperplane} that separates $z$ from $B$,  that is $w^{\top}z<w_0$. It remains to show that this inequality can be obtained `via dual multipliers': This is straightforward by again examining the conic program $\inf\{w^{\top}x\, | \ x\in T\}$ and applying Lemma~ \ref{lem:strdual}.
\end{proof}

%%
%%

%%%
%%%
\begin{proof} \emph{of Theorem~\ref{thm:PolyhedralOuterApproxOfBoundedConicCuts}}
Let $\pi^{\top}x \geq \pi_0$ be a valid inequality for $T^I$, where $\pi\in\mathbb{Z}^n$ is non-zero. Suppose $B= \{x\in T \,|\, \pi^{\top}x \leq \pi_0\}$ is nonempty and bounded. Then, by Proposition \ref{prop:PolyhedralOuterApproxOfBoundedConicCuts}, using dual multipliers $y^0,y^1,\dots,y^{p'} \in K^*$, and letting $P= \{x\in\mathbb{R}^n \,|\, \pi^{\top}x\leq \pi_0, \ (y^i)^{\top}A x \geq (y^i)^{\top}b, \ i\in [p']\}$, we have that 
(i) $P \supseteq B$ and (ii) $P \cap \mathbb{Z}^n = B \cap \mathbb{Z}^n$. Note that $\textup{interior}(B) \cap \mathbb{Z}^n= \emptyset$ and the only integer points in $B$ are those that satisfy $\pi^{\top}x = \pi_0$.

Now using an argument similar to Corollary 16.5a~\cite{Schrijverbook1}, there is a subset of $2^n$ inequalities defining $P$ together with $\pi^Tx < \pi_0$ such that the resulting set contains no integer points. WLOG $\{x\in\mathbb{R}^n \,|\, \ \pi^\top x \leq \pi_0, \  (y^i)^{\top}A x \geq (y^i)^{\top}b, \ i\in [p]\}$ is lattice-free, where $p \leq 2^n$, i.e., $\pi^{\top}x\geq \pi_0$ is a valid inequality for the integer hull of $Q = \{x\in\mathbb{R}^n \,|\, \ (y^i)^{\top}A x \geq (y^i)^{\top}b, \ i\in [p]\}$ where $(y^i)^{\top}A \in \mathbb{Q}^n$ for $i \in [p]$.
\end{proof}

\begin{remark}
If $T \cap \mathbb{Z}^n \neq \emptyset$, then using the same argument as in the proof of Corollary 16.6~\cite{Schrijverbook1} (also see~\cite{NYAS:NYAS284}), the bound of $2^n$ in Theorem~\ref{thm:PolyhedralOuterApproxOfBoundedConicCuts} can be improved to $2^{n}-1$.
\end{remark}

%%%%%%%%%%%%%%%%
%%%%%%%%%%%%%%%%%
%\section{Cannot construct appropriate polyhedral outer approximation of hyperbola}\label{sec:Tnon}

The next proposition illustrates that if the set $B$ in the statement of Theorem~\ref{thm:PolyhedralOuterApproxOfBoundedConicCuts} is not bounded, then the result may not hold.

\begin{proposition}\label{prop:NoGoodOuterApproximationForHyperbola}
Let $T' :=\{(x\in \mathbb{R}^2_{+}\,|\, x_1 x_2 \geq 1\}$. Every polyhedral outer approximation of $T'$ contains points of the form $(0,k)$ (and similarly points of form $(k,0)$) for $k$ sufficiently large natural number.
\end{proposition}
\begin{proof}
Suppose $\{x \in \mathbb{R}^2\,|\,\alpha^i_{1}x_1+\alpha^i_{2}x_2\geq \beta_i, \ i \in [q]\},$ is a polyhedral outer approximation of $T'$ where $q$ is some natural number. Since the recession cone of this polyhedron contains the recession cone of $T'$, that is $\mathbb{R}^2_{+}$, we have that $\alpha^i_{1}, \alpha^i_{2} \geq 0$. 

We will prove that there exist points of the form $(0,k)$ belonging to this outer approximation by showing that for all $i\in [q]$ there exists a $k_i$ such that $(\alpha^i)^{\top}(0, t) \geq \beta_i$ for all $t \in [k_i, \infty) \cap \mathbb{Z}$.  If $\alpha^i_{2} = 0$, then $\beta_i \leq 0$ (since $\alpha^i_{1}/k + \alpha^i_{2}k \geq \beta_i$ for all $k \in \mathbb{R}_{+}$). Therefore $k_i = 0$. If $\alpha^i_{2} > 0$, then $k_i = \beta_i/\alpha^i_{2}$.  \end{proof}

%%%%%%%%%%%%%
%%%%%%%%%%%%%
\section{A family of cut-generating functions in $\subfcts{\mathcal{L}^m}$ and its properties}
\label{sec:result2}

In this section, we show that $f_{\gamma}$ defined in (\ref{OurFormula}) belongs to $\subfcts{K}$.
Clearly $f_{\gamma}$ satisfies property (3.) in the definition of $\subfcts{K}$, that is $f_{\gamma}(0)=0$.  In Proposition~\ref{prop:OurFormulaIsSubadditive} and \ref{prop:OurFormulaIsNondecrLn} we prove that $f_{\gamma}$ also satisfies properties (1.) and (2.).

\begin{proposition}\label{prop:OurFormulaIsSubadditive}
The function $f_{\gamma}$ defined in (\ref{OurFormula}) is subadditive.
\end{proposition}
\begin{proof}
Let $u,v\in\mathbb{R}^m$. If at least one of these vectors fits in the first clause of (\ref{OurFormula}), then we have
$$f_{\gamma}(u+v)\leq \left\lceil\gamma^{\top}(u+v)\right\rceil +1 \leq \left\lceil\gamma^{\top}u\right\rceil + \left\lceil\gamma^{\top}v\right\rceil +1\leq f_{\gamma}(u)+f_{\gamma}(v).$$
Now, suppose that neither $u$ nor $v$ satisfies the first clause. If $u+v$ does not fit in the first clause, then we are done because $\lceil\cdot\rceil$ is a subadditive function. Assume $u+v$ satisfies the first clause, that is 
\begin{align}
u_j+v_j\neq 0, \ \ \gamma^{\top}(u+v)=\gamma^{\top}u+\gamma^{\top}v \in \mathbb{Z}.\label{eq:a1}
\end{align}
In this case, $u_j$ and $v_j$ cannot be simultaneously zero, say $u_j\neq 0$. Then
\begin{align}
\gamma^{\top}u &\notin \mathbb{Z},\label{eq:a2}
\end{align}
because $u$ does not satisfies the first clause.
It follows from (\ref{eq:a1}) and (\ref{eq:a2}) that
\begin{align}
\gamma^{\top}v &\notin \mathbb{Z}. \label{eq:a3}
\end{align}
Finally, (\ref{eq:a1}), (\ref{eq:a2}), (\ref{eq:a3}) together imply
\begin{align*}
f_{\gamma}(u)+f_{\gamma}(v)=\left\lceil\gamma^{\top}u\right\rceil + \left\lceil\gamma^{\top}v\right\rceil = \gamma^{\top}u+\gamma^{\top}v+1=f_{\gamma}(u+v),
\end{align*}
where the second inequality follows from the fact that $\gamma^{\top}u+\gamma^{\top}v\in\mathbb{Z}$.
\end{proof}

\begin{lemma}\label{lemma:StrictlyIneq}
Let $w\in\mathcal{L}^m$ and $j\in[m-1]$. Let $\Gamma_j$ be the set as in the statement of Theorem~\ref{thm:newfunction}. If $\gamma \in \mathcal{L}^m$, then $\gamma^{\top}w\geq 0$. If, in addition, $\gamma \in\Gamma_j\cup\interior{( \mathcal{L}^m)}$ and  $w_j\neq 0$, then $\gamma^{\top}w> 0$.
\end{lemma}
\begin{proof}
We have that $\gamma\in\mathcal{L}^m$. Therefore, since $w\succeq_{\mathcal{L}^m} 0$ and $\mathcal{L}^m$ is a self-dual cone, we conclude that $\gamma^{\top}w\geq 0$. Now, assume $w_j\neq 0$. If either $\gamma$ or $w$ is in the interior of $\mathcal{L}^m$, then it follows directly from Lemma~\ref{lemma:InteriorOfDualCone} that  $\gamma^{\top}w> 0$. Assume $\gamma,w\notin\interior{( \mathcal{L}^m)}$. Then
\begin{align}
w_m&=\sqrt{w_1^2+w_2^2+\dots+w_{m-1}^2} \label{eq:w}\\
\gamma_m&=\sqrt{\gamma_1^2+\gamma_2^2+\dots+\gamma_{m-1}^2}.\label{eq:gamma}
\end{align}
Two observations follows: (i) as $w_j\neq 0$, equation (\ref{eq:w}) implies that for all $i\in[m-1] $ such that $i\neq j$ we have
$w_m > |w_i|$; 
(ii) since $\gamma_m>|\gamma_j|$, equation (\ref{eq:gamma}) implies that $\gamma_i\neq 0$ for some $i\in[m-1] $ such that $i\neq j$.
Now, for all $i\in[m-1]$ such that $\gamma_i\geq 0$, we multiply $w_m > -w_i$ by $\gamma_i$ and, for all $i\in[m-1]$ such that $\gamma_i< 0$, we multiply $w_m > w_i$ by $-\gamma_i$. In view of observations (i) and (ii), at least one of the resulting inequalities remains strict. Then adding them all we obtain
\begin{align*}
\sum_{i\in[m-1]: \ \gamma_i\geq 0}\gamma_iw_m + \sum_{i\in[m-1]: \ \gamma_i< 0}-\gamma_iw_m &> \sum_{i\in[m-1]: \ \gamma_i\geq 0}\gamma_i (-w_i) + \sum_{i\in[m-1]: \ \gamma_i< 0}(-\gamma_i) w_i \\
\Rightarrow \ \sum_{i\in[m-1]}|\gamma_i|w_m &> -\sum_{i\in[m-1]}\gamma_iw_i \\
\Rightarrow \ \gamma_m w_m &> -\sum_{i\in[m-1]}\gamma_iw_i ,
\end{align*}
where the last implication follows from the fact that $ \gamma_m\geq \sum_{i=1}^{m-1}|\gamma_i|$ and $w_m\geq 0$. The result follows from this last inequality.
\end{proof}

\begin{proposition}\label{prop:OurFormulaIsNondecrLn}
The function $f_{\gamma}$ defined in (\ref{OurFormula}) is non-decreasing with respect to $\mathcal{L}^m$.
\end{proposition}
\begin{proof}
Let $u,v\in\mathbb{R}^m$. Suppose 
$u\succeq_{\mathcal{L}^m} v$. By applying Lemma~\ref{lemma:StrictlyIneq} to $w=u-v$ we conclude that
\begin{align}
\gamma^{\top}u\geq \gamma^{\top}v,\label{ineq:Conic}
\end{align}  
where the inequality (\ref{ineq:Conic}) holds strictly whenever $u_j-v_j\neq 0$.
Now, we use these facts to prove that 
$
f_{\gamma}(v)\leq f_{\gamma}(u).
$
If $u$ fits in the first clause of (\ref{OurFormula}), then
$
f_{\gamma}(v)\leq \gamma^{\top}v+1\leq \gamma^{\top}u+1=f_{\gamma}(u),
$
where the second inequality follows from (\ref{ineq:Conic}).
Assume $u$ does not satisfies the first clause. If $v$ does not fit in the first clause, then the result follows directly from (\ref{ineq:Conic}) and the fact that $\lceil\cdot\rceil$ is non-decreasing. Suppose $v$ satisfies the first clause, that is $v_j\neq 0$ and $\gamma^{\top}v\in\mathbb{Z}$. In this case, if $u_j= 0$, then $u_j-v_j\neq 0$ and hence (\ref{ineq:Conic}) holds strictly. Therefore, we conclude that
$
f_{\gamma}(v)=\gamma^{\top}v+1 \leq \lceil\gamma^{\top}u \rceil=f_{\gamma}(u).
$
On the other hand, if $u_j\neq 0$, then $\gamma^{\top}u\notin\mathbb{Z}$ (since $u$ does not satisfy the first clause), and using (\ref{ineq:Conic}) we obtain
$
\gamma^{\top}v < \lceil\gamma^{\top}u\rceil$ and hence $f_{\gamma}(v)=\gamma^{\top}v+1 \leq \lceil\gamma^{\top}u\rceil=f_{\gamma}(u),
$
which completes the proof.
\end{proof}
\color{black}

%\begin{remark}\label{remark:OurFunctionIsNotNondec.w.r.t.R3}
%The function $f$ defined in (\ref{OurFormula}) is \textbf{not} necessarily non-decreasing w.r.t. $\mathbb{R}_+^m$. Here is a counterexample in $\mathbb{R}^3$. Let $\gamma=(0,\tau,\tau)$ where $\tau$ is a positive constant. Thus $\gamma$ is in $\Gamma_1$. Then by choosing $j$ to be $1$ in the definition of $f_{\gamma}$ we have
%\begin{align}
%f_{\gamma}(x_1,x_2,x_3)= \begin{cases}
%\tau(x_2+x_3) +1 \ & \text{if} \ x_1\neq 0 \ \text{and} \ \tau(x_2+x_3)\in\mathbb{Z}, \vspace{0.1cm}\\
%\left\lceil\tau(x_2+x_3)\right\rceil    &\text{otherwise}.  
%\end{cases}\label{OurParticularFormula}
%\end{align}
%Consider the vectors  $p=(0,0,1/\tau)$ and $q=(-1,0,1/\tau)$. Then $p\geq_{\mathbb{R}^3_+}q$, but
%$$f_{\gamma}(p)=1<2=f_{\gamma}(q).$$ 
%\end{remark}
%
%One consequence of Remark \ref{remark:OurFunctionIsNotNondec.w.r.t.R3} is that $f_{\gamma}$ does not belong to the set of functions that generate Gomory, split, or disjunctive cuts for linear integer programming.

\section{Application of cut-generating functions in $\mathbb{R}^2$}\label{sec:result3}

In this section, we will prove Theorem~\ref{thm:2d}. We begin with proofs of two technical lemmas.
\begin{lemma}\label{lemma:ParabolaUnboundedProblem}
Let $W^{i}=\{x\in\mathbb{R}^2\, | \ A^ix\succeq_{\mathcal{L}^{m_i}}b^i\}$ be a parabola, where 
$A^{i}\in\mathbb{R}^{m_{i}\times 2}$, $b^{i}\in\mathbb{R}^{m_{i}}$ and $\mathcal{L}^{m_{i}}$ is the second-order cone in $\mathbb{R}^{m_{i}}$. If $\pi\in\recdual{W^{i}}\setminus\interior{(\recdual{W^{i}})}$, $\pi\neq 0$, then the problem
$
\inf \{\pi^{\top}x\, | \ x\in W^{i}\}
$
is unbounded.
\end{lemma}
\begin{proof}
Up to a rotation, any parabola in $\mathbb{R}^2$ can be written as
$
\{(x,y)\in\mathbb{R}^2\,| \ y\geq \rho(x-x_0)^2+y_0\},
$
where $\rho> 0$. In this case, the recession cone of the parabola is a vertical line. As $\pi\in \recdual{W^{i}}\setminus\interior{(\recdual{W^{i}})}$ we must have  $\pi_2=0$, in which case $\pi_1\neq 0$ and the problem is clearly unbounded.
\end{proof}

\begin{lemma}\label{lemma:OnlyOneConicIsEnough}
Let $W$ be the set as in the statement of Theorem~\ref{thm:2d}. Assume, in addition, that $W$ is unbounded. Let $\pi\neq 0$ be such that $\pi\notin \interior{(\recdual{W})}$. If the problem
\begin{align}
\alpha := \inf \{\pi^{\top} x\, | \ x \in W\}\label{primal}
\end{align}
is bounded, then there exists $i_0\in[m]$ such that
\begin{align}
\alpha = \inf \{\pi^{\top}x\, | \ W^{i_0}\}.\label{primal2}
\end{align}
Moreover, $W^{i_0}=\{x\in\mathbb{R}^2\, | \ A^{i_0}x\succeq_{\mathcal{L}^{m_{i_0}}}b^{i_0}\}$ is either:\\
(i) a half-space defined by $\pi^{\top}x\geq \alpha$; or\\ 
(ii) one branch of a hyperbola whose one of the asymptotes is orthogonal to $\pi$.
\end{lemma}
\begin{proof}
Since the primal problem (\ref{primal}) is bounded and strictly feasible, we have 
that its dual
\begin{align}
\sup\{\sum_{i=1}^m(b^i)^{\top}y^i\, | \ \sum_{i=1}^m (y^i)^{\top}A^i = \pi^{\top}, \ y^i\in \mathcal{L}^*_{m_i} \ \forall i \in [m]\}\label{dual}
\end{align}
is solvable \cite{NemirovskiNotes}. We will show that (\ref{dual}) admits an optimal solution for which $y^i=0$ for all $i\in [m]$ except for one particular $i_0\in [m]$.

Since (\ref{primal}) is bounded, it follows from Lemma~\ref{lemma:ConditionBoundedConicProgram} that $\pi\in \recdual{W}$. On the other hand, by assumption $\pi$ is not in the interior of that cone.  Therefore, using Lemma~\ref{lemma:InteriorOfDualCone} we conclude that there exists a non-zero vector $d_0\in \rec{W}$ such that
$
\pi^{\top}d_0 = 0.
$
Then any feasible solution $(y^1,y^2, \cdots, y^m)$ of (\ref{dual}) satisfies
\begin{align*}
0=\pi^{\top}d_0 = \sum_{i=1}^m (y^i)^{\top}A^id_0.
\end{align*}
Moreover, each term in this summation is non-negative since $A^id_0 \succeq_{\mathcal{L}_{m_i}} 0$ (recall $d_0\in \rec{W}$) and $y^i\in\mathcal{L}^*_{m_i}$, for all $i\in [m]$. As a result, we have
$
(y^i)^{\top}A^id_0 = 0 \ \forall i \in [m].
$
As $d_0$ is a non-zero vector in $\mathbb{R}^2$, we conclude that for each $i\in [m]$ there must exist a scalar $\lambda_i$ such that
\begin{align}
(y^i)^{\top}A^i = \lambda_i \pi^{\top}.\label{eq:pi}
\end{align}
We claim that $\lambda_i\geq 0$ for all $i\in [m]$. To prove the claim, all we need to show is that $(y^i)^{\top}A^i$ and $\pi$ are in the same half-space. By assumption $\pi\in\recdual{W}$. Since $\recdual{W}$ is contained in a half-space (otherwise we would have $\rec{W}=\{0\}$ which contradicts the fact that $W$ is unbounded), it is enough to prove that $(y^i)^{\top}A^i\in\recdual{W}$. To see why this is true, note that for all $d\in\rec{W^i}$ we have
$A^id \succeq_{\mathcal{L}_{m_i}} 0$, which implies $(y^i)^{\top}A^id \geq 0$. Thus, $(y^i)^{\top}A^i\in\recdual{W^i}\subseteq\recdual{W}$, where the last containment follows from the fact that $\rec{W^i} \supseteq \rec{W}$.

Now, suppose $(y^1,y^2, \cdots, y^m)$ is an optimal solution of the dual problem (\ref{dual}). If $\lambda_i=0$, then we must have $(b^i)^{\top}y^i = 0$, because $(b^i)^{\top}y^i > 0$ would imply the dual problem to be unbounded and $(b^i)^{\top}y^i<0$ would imply that the current solution is not optimal. Hence we have that if $\lambda_i=0$, then we can set $y^i=0$ without altering the objective value. On the other hand, since $\pi\neq 0$, (\ref{eq:pi}) combined with the equality in (\ref{dual}) imply that  the $\lambda$'s add up to $1$. Thus, we cannot have $\lambda_i=0$ for all $i\in[m]$. Suppose $\lambda_i,\lambda_j>0$ for some $i,j\in[m]$, $i\neq j$. We claim that $(b^i)^{\top}y^i=(\lambda_i/\lambda_j)(b^j)^{\top}y^j$.
Without loss of generality, assume by contradiction that $(b^i)^{\top}y^i<(\lambda_i/\lambda_j)(b^j)^{\top}y^j$.
Then, since $\lambda_i+\lambda_j\leq 1$ we obtain
\begin{align*}
(b^i)^{\top}y^i+(b^j)^{\top}y^j<\frac{\lambda_i}{\lambda_j}(b^j)^{\top}y^j+(b^j)^{\top}y^j\leq \frac{1}{\lambda_j}(b^j)^{\top}y^j.
\end{align*}
In this case, we could set $\lambda_i=0, \ \lambda_j=1$ and $y^i=0$ to obtain a new feasible solution with objective value strictly larger. But this contradicts the fact that $y$ is an optimal solution. Thus, the claim holds and by setting $\lambda_i=0,\lambda_j=1$ and $y^i=0$ we obtain a new feasible solution with the same objective value, and hence optimal. In this case, we set $i_0=j$.

Consider now the primal-dual pair
\begin{align}
\beta:=&\inf\{\pi^{\top}x\, | \ \ A^{i_0}x\succeq_{\mathcal{L}_{m_{i_0}}} b^{i_0} \},\label{newprimal}\\
&\sup\{(b^{i_0})^{\top}y^{i_0}\, | \ (y^{i_0})^{\top}A^{i_0} = \pi^{\top}, \ y^{i_0}\in \mathcal{L}^*_{m_{i_0}}\}.\label{newdual}
\end{align}
Let $x^*$ be an $\varepsilon$-optimal solution to the original primal (\ref{primal}), that is $x^*\in W$ and $\pi^{\top}x^*\leq \alpha+\varepsilon$. Clearly, $x^*$ is feasible for (\ref{newprimal}). Note now that the dual solution constructed above for (\ref{dual}), when restricted to the $y^{i_0}$ component is a feasible solution to (\ref{newdual}) with objective value $\alpha$. Thus, we have $\alpha \leq \beta \leq \pi^{\top}x^*\leq \alpha+\varepsilon$, where the first inequality follows from weak duality to the primal-dual pair (\ref{newprimal}-\ref{newdual}) and the second inequality follows from fisibility of $x^*$ to (\ref{newprimal}). By taking the limit as $\varepsilon$ goes to zero, we obtain (\ref{primal2}).

To prove the second part of the lemma, we first observe that $\recdual{W^{i_0}}\subseteq \recdual{W}$. If $\pi\notin \recdual{W^{i_0}}$, then (\ref{primal2}) would be unbounded by Lemma~\ref{lemma:ConditionBoundedConicProgram}. As $\pi\notin \interior{(\recdual{W})}$, we have that $\pi\notin \interior{(\recdual{W^{i_0}})}$. Hence, $\pi\in \recdual{W^{i_0}}\setminus\interior{(\recdual{W^{i_0}})}.$

Now, $W^{i_0}$ cannot define an ellipse because then $W\subseteq W^{i_0}$ would be bounded.
Since $\pi\in\recdual{W^{i_0}}\setminus\textup{interior}(\recdual{W^{i_0}})$, if $W^{i_0}$ was a parabola, then problem (\ref{primal2}) would be unbounded in view of Lemma~\ref{lemma:ParabolaUnboundedProblem}. Therefore, only two possibilities remain:\\
(i) $W^{i_0}$ is defined by a linear inequality, say $\mu^{\top}x\geq \mu_0$. In this case $\mu$ must be a multiple of $\pi$, otherwise problem (\ref{primal2}) would be unbounded. Thus, we may assume $\pi=\mu$ and then $\mu_0=\alpha$.\\
(ii) $W^{i_0}$ is one branch of a hyperbola. In this case, $\rec{W^{i_0}}$ is defined by the asymptotes of the hyperbola. As $\pi\in\recdual{W^{i_0}}\setminus\textup{interior}(\recdual{W^{i_0}})$,  $\pi$  must be orthogonal to one of the asymptotes.
\end{proof}

Next we prove Lemma~\ref{lemma:ConicRepresentationHyperbola} that was stated in Section~\ref{sec:cuts}.
\begin{proof}~\emph{of Lemma~\ref{lemma:ConicRepresentationHyperbola}}
Any conic section (parabola, ellipse, hyperbola) in $\mathbb{R}^2$ is a curve defined by a quadratic equation of the form
\begin{align}
\frac{1}{2}x^{\top}Qx+d^{\top}x+s=0,\label{eq:hyperbola}
\end{align}
where $s$ is a scalar, $d\in\mathbb{R}^2$ and $Q=VDV^{\top}$. In this factorization, $V\in\mathbb{R}^{2\times 2}$ is orthonormal and 
$$D=\begin{bmatrix}
\lambda_1 & 0 \\ 
0 & \lambda_2
\end{bmatrix}, $$
where $\lambda_1, \lambda_2$ are the eigenvalues of $Q$. In particular, the curve defined by (\ref{eq:hyperbola}) is a hyperbola if and only if one of these eigenvalues is positive and the other is negative. After changing variables
$
y:=V'x
$
and completing squares, equation (\ref{eq:hyperbola}) can be written in exactly one of the following forms
\begin{align}
%[\sqrt{\lambda_1}(y_1-c_1/\lambda_1)]^2-[\sqrt{|\lambda_2|}(y_2-c_2/\lambda_2)]^2=\frac{c_1^2}{\lambda_1}+\frac{c_2^2}{\lambda_2}-2\gamma,\\
%where \ c_1:=d_1v_{11}+d_2v_{21}, \ c_2:=d_1v_{12}+d_2v_{22}\\
%
[\beta_1(y_1-\alpha_1)]^2-[\beta_2(y_2-\alpha_2)]^2=\pm\eta^2,
\label{eq:CanonicHyperbola}
\end{align}
where $\eta$ and $\alpha_i, \beta_i$, for $i=1,2$, are constants depending on the coefficients of (\ref{eq:hyperbola}). In what follows, we assume that the coefficient of $\eta^2$ is positive. If it was negative, then we could multiply (\ref{eq:CanonicHyperbola}) by $-1$ and all we will do next would be analogous.
Under this assumption, one branch of the hyperbola is given by
\begin{align*}
G^+&:=\{y\in\mathbb{R}^2\, | \ (\eta)^2+[\beta_2(y_2-\alpha_2)]^2\leq [\beta_1(y_1-\alpha_1)]^2, \ \beta_1(y_1-\alpha_1)\geq 0\}\\
&=\{y\in\mathbb{R}^2\, | \ \sqrt{\eta^2+[\beta_2(y_2-\alpha_2)]^2}\leq \beta_1(y_1-\alpha_1)\}\\
&=\{y\in\mathbb{R}^2\, | \ (\eta,\beta_2(y_2-\alpha_2),\beta_1(y_1-\alpha_1))\in\mathcal{L}^3\}\\
&=\{y\in\mathbb{R}^2\, | \ \begin{bmatrix}
0 & 0 \\ 
0 & \beta_2 \\ 
\beta_1 & 0
\end{bmatrix}
\begin{bmatrix}
y_1 \\ 
y_2
\end{bmatrix}
\succeq_{\mathcal{L}^3}\begin{bmatrix}
-\eta \\ 
\beta_2\alpha_2 \\
\beta_1\alpha_1
\end{bmatrix}\}.
\end{align*}
Then, going back to the space of the original variables we obtain
\begin{align*}
G^+=\{x\in\mathbb{R}^2\, | \ \begin{bmatrix}
0 & 0 \\ 
\beta_2v_{12} & \beta_2v_{22} \\ 
\beta_1v_{11} & \beta_1v_{21}
\end{bmatrix}
\begin{bmatrix}
x_1 \\ 
x_2
\end{bmatrix}
\succeq_{\mathcal{L}^3}\begin{bmatrix}
-\eta \\ 
\beta_2\alpha_2 \\
\beta_1\alpha_1
\end{bmatrix}\},
\end{align*}
where $v_{ij}$ are the entries of the matrix $V$.
The other branch of the hyperbola is given by
\begin{align*}
G^-&:=\{y\in\mathbb{R}^2\, | \ (\eta)^2+[\beta_2(y_2-\alpha_2)]^2\leq [\beta_1(y_1-\alpha_1)]^2, \ \beta_1(y_1-\alpha_1)\leq 0\}.
\end{align*}
After the change of variables
$\tilde{y}:=-y$
we obtain
\begin{align*}
G^-&=\{\tilde{y}\in\mathbb{R}^2\, | \ (\eta)^2+[\beta_2(-\tilde{y}_2-\alpha_2)]^2\leq [\beta_1(-\tilde{y}_1-\alpha_1)]^2, \ \beta_1(-\tilde{y}_1-\alpha_1)\leq 0\}\\
&=\{\tilde{y}\in\mathbb{R}^2\, | \ (\eta)^2+[\beta_2(\tilde{y}_2+\alpha_2)]^2\leq [\beta_1(\tilde{y}_1+\alpha_1)]^2, \ \beta_1(\tilde{y}_1+\alpha_1)\geq 0\}\\
&=\{\tilde{y}\in\mathbb{R}^2\, | \ (\eta,\beta_2(\tilde{y}_2+\alpha_2),\beta_1(\tilde{y}_1+\alpha_1))\in\mathcal{L}^3\}\\
&=\{\tilde{y}\in\mathbb{R}^2\, | \ \begin{bmatrix}
0 & 0 \\ 
0 & \beta_2 \\ 
\beta_1 & 0
\end{bmatrix}
\begin{bmatrix}
\tilde{y}_1 \\ 
\tilde{y}_2
\end{bmatrix}
\succeq_{\mathcal{L}^3}\begin{bmatrix}
-\eta \\ 
-\beta_2\alpha_2 \\
-\beta_1\alpha_1
\end{bmatrix}\}.
\end{align*}
Going back to the space of the original variables we obtain
\begin{align*}
G^-=\{x\in\mathbb{R}^2\, | \ \begin{bmatrix}
0 & 0 \\ 
-\beta_2v_{12} & -\beta_2v_{22} \\ 
-\beta_1v_{11} & -\beta_1v_{21}
\end{bmatrix}
\begin{bmatrix}
x_1 \\ 
x_2
\end{bmatrix}
\succeq_{\mathcal{L}^3}\begin{bmatrix}
-\eta \\ 
-\beta_2\alpha_2 \\
-\beta_1\alpha_1
\end{bmatrix}\}.
\end{align*}
It follows from (\ref{eq:CanonicHyperbola}) that the asymptotes of $G^+$ have equations
\begin{align*}
\beta_1y_1+\beta_2y_2&=\beta_1\alpha_1+\beta_2\alpha_2,\\
\beta_1y_1-\beta_2y_2&=\beta_1\alpha_1-\beta_2\alpha_2.
\end{align*}
In the space of $x$ variables they become
\begin{align}
(\beta_1v_{11}+\beta_2v_{12})x_1+(\beta_1v_{21}+\beta_2v_{22})x_2&=\beta_1\alpha_1+\beta_2\alpha_2,\label{eq:asymptotes1}\\
(\beta_1v_{11}-\beta_2v_{12})x_1+(\beta_1v_{21}-\beta_2v_{22})x_2&=\beta_1\alpha_1-\beta_2\alpha_2.\nonumber %\label{eq:asymptotes2}
\end{align}
The asymptotes of $G^-$ are obtained in a similar way.
\end{proof}

\begin{lemma}\label{lemma:CustForGeneralHyperbola}
Let $G$ be one branch of a non-degenerate hyperbola in $\mathbb{R}^2$. Let $\pi^{\top}x\geq \pi_0$ be a face of $G^I$ such that $\pi\in\mathbb{Z}^2$ is non-zero and orthogonal to one of the asymptotes. Then $\pi^{\top}x\geq \pi_0$ is one of the inequalities (\ref{CutsForHyperbola}).
\end{lemma}
\begin{proof}
Using the same notation adopted in the proof of Lemma~\ref{lemma:ConicRepresentationHyperbola} above, we assume $G=G^+$. If $G=G^-$, then the proof is analogous. Note that $G$ is contained in the set
\begin{align*}
H:=\{x\in\mathbb{R}^2\, | \ (\beta_1v_{11}+\beta_2v_{12})x_1+(\beta_1v_{21}+\beta_2v_{22})x_2&\geq\beta_1\alpha_1+\beta_2\alpha_2, \\ 
(\beta_1v_{11}-\beta_2v_{12})x_1+(\beta_1v_{21}-\beta_2v_{22})x_2&\geq\beta_1\alpha_1-\beta_2\alpha_2\}.
\end{align*}
Assume $\pi$ is orthogonal to the asymptote (\ref{eq:asymptotes1}). The proof of the case in which $\pi$ is orthogonal to the second asymptote is similar. Since $\pi\in\mathbb{Z}^2$ is non-zero, we may assume that the coefficients of $x_1$ and $x_2$ in (\ref{eq:asymptotes1}) are integers. Let 
$$\tau:=\gcd\{\beta_1v_{11}+\beta_2v_{12},\beta_1v_{21}+\beta_2v_{22}\}.$$
Since the hyperbola is non-degenerate, the line
\begin{align*}
(\beta_1v_{11}+\beta_2v_{12})x_1+(\beta_1v_{21}+\beta_2v_{22})x_2= \beta_1\alpha_1+\beta_2\alpha_2
\end{align*}
does not intersect $G$. However,  for all $\varepsilon>0$, the equation
\begin{align}
\frac{\beta_1v_{11}+\beta_2v_{12}}{\tau}x_1+\frac{\beta_1v_{21}+\beta_2v_{22}}{\tau}x_2= \frac{\beta_1\alpha_1+\beta_2\alpha_2}{\tau}+\varepsilon\label{eq:LatticeFree}
\end{align}
intersects $G$ along a ray. Moreover, (\ref{eq:LatticeFree}) has integral solutions if and only if the right-hand-side is integral.

Therefore, if $(\beta_1\alpha_1+\beta_2\alpha_2)/\tau\in \mathbb{Z}$, then the inequality 
\begin{align}
\frac{\beta_1v_{11}+\beta_2v_{12}}{\tau}x_1+\frac{\beta_1v_{21}+\beta_2v_{22}}{\tau}x_2\geq  \frac{\beta_1\alpha_1+\beta_2\alpha_2}{\tau}+1\label{ineq:FirstFacet}
\end{align}
is a face of $G^I$, and hence it is equivalent to $\pi^{\top}x\geq \pi_0$. On the other hand, if  $(\beta_1\alpha_1+\beta_2\alpha_2)/\tau \notin \mathbb{Z}$, then
\begin{align}
\frac{\beta_1v_{11}+\beta_2v_{12}}{\tau}x_1+\frac{\beta_1v_{21}+\beta_2v_{22}}{\tau}x_2\geq  \left\lceil\frac{\beta_1\alpha_1+\beta_2\alpha_2}{\tau}\right\rceil\label{ineq:SecondFacet}
\end{align}
is a face of $G^I$, and hence it is equivalent to $\pi^{\top}x\geq \pi_0$. 
%See Figure \ref{fig:BothKindOfInterestingCuts}.
%\begin{figure}
%\begin{center}
%\includegraphics[scale=0.42]{figures/HyperbolaCuts.pdf} 
%\end{center}
%\caption{On the left-hand-side figure the red line represents a cut of the form (\ref{ineq:FirstFacet}). On the left-hand-side figure the green line represents a cut of the form (\ref{ineq:SecondFacet}).} \label{fig:BothKindOfInterestingCuts}
%\end{figure}

Observe now that (\ref{ineq:FirstFacet}) and (\ref{ineq:SecondFacet}) are one of the inequalities (\ref{CutsForHyperbola}) in view of Proposition~\ref{prop:1}.
\end{proof}

Next we use Lemma~\ref{lemma:OnlyOneConicIsEnough} and Lemma~\ref{lemma:CustForGeneralHyperbola} above to proof Theorem~\ref{thm:2d}.
\begin{proof}~\emph{of Theorem~\ref{thm:2d}} 
First, we observe that if $W$ is bounded, then the result follows directly from Theorem~\ref{thm:PolyhedralOuterApproxOfBoundedConicCuts}.
Suppose $W$ is unbounded. We have two cases:

\noindent\textbf{Case 1:} $W\cap \mathbb{Z}^2 = \emptyset$. 
In this case, there exist $\pi=(\pi_1,\pi_2)$ with $\pi_1, \pi_2$ integer relatively prime and a integer $\pi_0$ such that \cite{Dey2010, BasuCCZ10}
\begin{align}
W\subseteq \{x\in\mathbb{R}^2\, | \ \pi_0 \leq \pi^{\top}x \leq\pi_0+1\}.\label{854}
\end{align}
We will show that the cut $\pi^{\top} x  \geq  \pi_0+1$ can be obtained using subadditive functions (\ref{Linearcomposite}) or using one of the inequalities (\ref{CutsForHyperbola}). Analogous proof holds for the cut $\pi^{\top} x  \leq \pi_0$.
A consequence of $W$ being between these two lines is that $\rec{W}$ is orthogonal to $\pi$ and, therefore, $\pi\notin\interior{(\recdual{W})}$ in view of Lemma~\ref{lemma:InteriorOfDualCone}. Then, by Lemma~\ref{lemma:OnlyOneConicIsEnough},
\begin{align*}
\alpha:=\inf \{\pi^{\top}x\, | \ W^{i_0}\}=\inf \{\pi^{\top}x\, | \ x\in W\},
\end{align*}
for some $i_0\in[m]$, where there are only two possibilities for $W^{i_0}=\{x \in\mathbb{R}^2 \,| \  A^{i_0}x\succeq_{\mathcal{L}^{m_{i_0}}}b^{i_0}\}$:\\
\textbf{\textit{(i)}} $W^{i_0}$ is the half-space $\pi^{\top}x\geq \alpha$: In this case, since $A^{i_0}x\succeq_{\mathcal{L}^{m_{i_0}}}b^{i_0}$ is non-redundant, we have that the line $\pi^{\top}x=\alpha$ intersects $W$.
Note that $\pi_0\leq \alpha$ in view of (\ref{854}). Since $W$ is unbounded and its recession cone is orthogonal to $\pi$, if $\alpha=\pi_0$, then $W$ would contain a integer point from the line $\pi^{\top}x=\pi_0$. Therefore, $\alpha>\pi_0$ in which case $\pi^{\top}x\geq \lceil \alpha\rceil=\pi_0+1$ is a valid inequality for $W^I$ and this cut can be obtained using a subadditive function (\ref{Linearcomposite}).\\
\textbf{\textit{(ii)}} $W^{i_0}$ is a hyperbola whose one of the asymptotes is orthogonal to $\pi$: Without loss of generality, we may assume that the asymptote orthogonal to $\pi$ has equation $\pi^{\top}x=\alpha$. Let 
\begin{align}
\beta=\begin{cases}
\alpha+1 \ \text{if} \ \alpha\in\mathbb{Z}\\
\lceil\alpha\rceil \ \text{if} \ \alpha\notin\mathbb{Z}.
\end{cases}\label{beta}
\end{align}
Since the hyperbola is non-degenerate, we have that $\pi^{\top}x\geq \beta$ is a valid inequality for $(W^{i_0})^I$. Moreover, $\pi^{\top}x= \beta$ contains a ray of $W^{i_0}$ since $\beta>\alpha$. Then, since $\pi_1$ and $\pi_2$ are relatively prime and $\beta\in\mathbb{Z}$, we have that $\pi^{\top}x\geq \beta$ is, in addition, a face of $(W^{i_0})^I$. Now, it follows from Lemma~\ref{lemma:CustForGeneralHyperbola} that this face is one of the inequalities (\ref{CutsForHyperbola}). Finally, note that $\pi_0\leq \alpha< \pi_0+1$. Thus, we have that $\beta=\pi_0+1$.

\noindent\textbf{Case 2:} $\textup{interior}(W)\cap \mathbb{Z}^2 \neq \emptyset$. 
By assumption, the components of $\pi$ are integers and, without loss of generality, we may also assume they are relatively prime.
We now have three cases.
\begin{enumerate}
\item $\pi\notin\recdual{W}$: In this case, by Lemma~\ref{lemma:ConditionBoundedConicProgram}, we have that $\inf \{\pi^{\top} x\, | \ x \in W\}$ is unbounded. Since we assume that $\textup{interior}(W)\cap \mathbb{Z}^2 \neq \emptyset$, we obtain that  $\inf \{\pi^{\top} x\, | \ x \in W\cap \mathbb{Z}^2\}$ is unbounded \cite{DiegoSantanu2013}, which contradicts the fact that $\pi^{\top}x\geq \pi_0$ is a valid inequality for $W^I$.

\item $\pi\in\textup{interior}(\recdual{W})$ : 
In this case, $\{x\in W \, | \ \pi^{\top} x \leq \pi_0 \}$
is bounded in view of Proposition~\ref{prop:PolyhedralOuterApproxOfBoundedConicCuts}. Therefore, it follows from Theorem~\ref{thm:PolyhedralOuterApproxOfBoundedConicCuts} that the valid inequality $\pi^{\top}x\geq \pi_0$ can be obtained using functions (\ref{Linearcomposite}).

\item $\pi\in \recdual{W}\setminus\textup{interior}(\recdual{W})$:
%Let 
%$E:=\{\pi^{\top} x \leq \pi_0 \}\cap W$.
%If $E$ is bounded then the inequality $\pi^{\top}x\leq \pi_0$ can be obtained using functions (\ref{Linearcomposite}) in view of Theorem~\ref{thm:PolyhedralOuterApproxOfBoundedConicCuts}.
%So we assume that $E$ is unbounded. 
Since $\textup{interior}(W)\cap \mathbb{Z}^2 \neq \emptyset$ and $\inf \{\pi^{\top} x\, | \ x \in W\cap \mathbb{Z}^2\}$ is bounded, we have that $\alpha:=\inf \{\pi^{\top} x\, | \ x \in W\}$ is bounded \cite{DiegoSantanu2013}. Then, by Lemma~\ref{lemma:OnlyOneConicIsEnough},
$
\alpha=\inf \{\pi^{\top}x\, | \ W^{i_0}\},
$
for some $i_0\in[m]$, where there are only two possibilities for $W^{i_0}=\{x \in\mathbb{R}^2 \,| \  A^{i_0}x\succeq_{\mathcal{L}^{m_{i_0}}}b^{i_0}\}$:\\
\textbf{\textit{(i)}} $W^{i_0}$ is the half-space $\pi^{\top}x\geq \alpha$: Since $A^{i_0}x\succeq_{\mathcal{L}^{m_{i_0}}}b^{i_0}$ is non-redundant, we have that the line $\pi^{\top}x=\alpha$ intersects $W$. Thus, $\pi^{\top}x\geq \lceil \alpha\rceil$ is a valid inequality for $W^I$ and this cut can be obtained using a subadditive function (\ref{Linearcomposite}). Now, we only need to show that $\lceil\alpha\rceil=\pi_0$. It is enough to show that the line $\pi^{\top}x= \lceil \alpha\rceil$ intersects $W \cap \mathbb{Z}^2$. Note that the line $\pi^{\top}x=\lceil\alpha\rceil$ intersects $W$ (otherwise we would have $W\subseteq \{x\in\mathbb{R}^2\, | \ \pi^{\top}x < \lceil\alpha\rceil\}$ which contradicts the fact that $W\cap \mathbb{Z}^2\neq \emptyset$ since $\pi^{\top}x\geq \lceil\alpha\rceil$ is valid inequality for $W^I$).
Thus, $\{x\in W \, | \ \pi^{\top}x=\lceil \alpha\rceil\}\neq \emptyset$. Moreover, since $\pi\in \recdual{W}\setminus\textup{interior}(\recdual{W})$, there exists a non-zero vector $d\in\rec{W}$ such that $\pi^{\top}d=0$. Therefore, $d$ is in the recession cone of $\{x\in W \, | \ \pi^{\top}x=\lceil \alpha\rceil\}$. Hence, $\pi^{\top}x=\lceil \alpha\rceil$ contains a ray of $W$. Thus,  $\pi^{\top}x=\lceil\alpha\rceil$ contains an integer point of $W$ since $\pi_1$ and $\pi_2$ are relatively prime.\\
\textbf{\textit{(ii)}} $W^{i_0}$ is a hyperbola one of whose asymptotes is orthogonal to $\pi$: As in Case 1 \textit{(ii)}, we can show that $\pi^{\top}x\geq \beta$ is a face of $W^{i_0}$, where $\beta$ is defined in (\ref{beta}). Moreover, by Lemma~\ref{lemma:CustForGeneralHyperbola}, $\pi^{\top}x\geq \beta$ is one of the inequalities (\ref{CutsForHyperbola}). Now, only remains to show that $\beta=\pi_0$. It is enough to show that $\pi^{\top}x= \beta$ intersects $W \cap \mathbb{Z}^2$. Clearly, $\pi^{\top}x\geq\beta$ is a valid inequality for $W^I\subseteq W^{i_0}$. Since $\alpha<\beta$, we have that the line $\pi^{\top}x= \beta$ intersects $W$ (otherwise we would have $W\subseteq \{x\in\mathbb{R}^2\, | \ \pi^{\top}x < \beta\}$ which contradicts the fact that $W\cap \mathbb{Z}^2\neq \emptyset$). Therefore, as in the case \textit{(i)} above, we can prove that $\pi^{\top}x = \beta$ contains a ray of $W$. Thus, $\pi^{\top}x=\beta$ contains an integer point of $W$ since $\pi_1$ and $\pi_2$ are relatively prime and $\beta\in\mathbb{Z}$.
\end{enumerate}
\end{proof}

\section*{Acknowledgments}

The authors thank to two anonymous referees for their helpful corrections and suggestions.

Funding: This work was supported by the NSF CMMI [grant number 1149400]; and the CNPq [grant number 248941/2013-5].

%Santanu S. Dey acknowledges the support from NSF CMMI Grant 1149400. Asteroide Santana acknowledges the support from CNPq Grant 248941/2013-5.

\section*{References}

\bibliography{SOCP_Sub_Funct_bibfile}

\end{document}